\documentclass{amsart}
\usepackage[all,cmtip]{xy}
\usepackage{enumerate, comment}
\usepackage{mathrsfs}
\usepackage{ amssymb, latexsym, amsmath}
\usepackage{fullpage}
\usepackage{url}
\usepackage{hyperref}
\usepackage{helvet}
\usepackage{amsfonts}
\usepackage{tikz}
\usepackage{commath}
\usepackage{amsthm}
\usepackage{amscd}
\usepackage{stmaryrd}
\usepackage[mathscr]{eucal}
\usepackage{indentfirst}
\usepackage{graphicx}
\usepackage{tikz-cd}
\usepackage{graphics}
\usepackage{pict2e}
\usepackage{epic}
\usepackage{color}
\usepackage{mathtools}
\usepackage[normalem]{ulem}
\usepackage[OT2,T1]{fontenc}

\DeclarePairedDelimiter\floor{\lfloor}{\rfloor}

\newcommand{\closure}[2][3]{%
{}\mkern#1mu\overline{\mkern-#1mu#2}}
\numberwithin{equation}{section}
\usepackage[margin=3.4cm]{geometry}
 
\DeclareSymbolFont{cyrletters}{OT2}{wncyr}{m}{n}
\DeclareMathSymbol{\Sha}{\mathalpha}{cyrletters}{"58}

\newcommand{\edv}{\mathrel\Vert}

\newcommand{\fp}{\mathfrak{p}}
\newcommand{\cE}{\mathcal{E}}

\newcommand{\cyc}{\op{cyc}}

\newcommand{\bad}{\op{bad}}
\newcommand{\good}{\op{good}}
\newcommand{\ord}{\op{ord}}
\newcommand{\corank}{\op{corank}}
\newcommand{\tr}{\op{tr}}
\newcommand{\Gal}{\op{Gal}}
\newcommand{\Sel}{\op{Sel}}

\newcommand{\Z}{\mathbb{Z}}
\newcommand{\Q}{\mathbb{Q}}
\newcommand{\F}{\mathbb{F}}

\newcommand{\mup}{\mu_p(E)}
\newcommand{\lambdap}{\lambda_p(E)}

\newcommand{\Selp}{\Sel_{p^{\infty}}(E/\Q^{\cyc})}

\newcommand{\Selpm}{\Sel_{p^{\infty}}^{\pm}(E/\Q^{\cyc})}
\newcommand{\Selast}{\Sel_{p^{\infty}}^{\ddag}(E/\Q^{\cyc})}
\newcommand{\op}[1]{\operatorname{#1}}

\theoremstyle{plain}
 \theoremstyle{definition}
\newtheorem{Th}{Theorem}[section]
\newtheorem{Lemma}[Th]{Lemma}

\newtheorem{Corollary}[Th]{Corollary}
\newtheorem{Proposition}[Th]{Proposition}
\newtheorem{Remark}[Th]{Remark}
 \theoremstyle{definition}
\newtheorem{Definition}[Th]{Definition}
\newtheorem{Conjecture}[Th]{Conjecture}
\newtheorem*{lconj}{Conjecture}

\begin{document}

\title{Statistics for Iwasawa invariants of elliptic curves}
\author{Debanjana Kundu}
\address{Department of Mathematics \\ University of British Columbia \\
  Vancouver BC, V6T 1Z2, Canada.} 
  \email{dkundu@math.ubc.ca}
\author{Anwesh Ray}
\address{Department of Mathematics \\ University of British Columbia \\
  Vancouver BC, V6T 1Z2, Canada.} 
\email{anweshray@math.ubc.ca}


\date{\today}


\begin{abstract}
We study the average behaviour of the Iwasawa invariants for the Selmer groups of elliptic curves, setting out new directions in arithmetic statistics and Iwasawa theory. \end{abstract}

\maketitle
\section{Introduction}
Iwasawa theory began as the study of class groups over infinite towers of number fields.
In \cite{mazur72}, B. Mazur initiated the study of Iwasawa theory of elliptic curves.
The main object of study is the $p$-primary Selmer group of an elliptic curve $E$, taken over the cyclotomic $\Z_p$-extension of $\Q$.
Mazur conjectured that when $p$ is a prime of good ordinary reduction, the $p$-primary Selmer group is cotorsion as a module over the Iwasawa algebra, denoted by $\Lambda$.
This conjecture was settled by K. Kato, see \cite[Theorem 17.4]{Kat04}.
\par
Note that the Iwasawa algebra $\Lambda$ is isomorphic to the power series ring $\Z_p\llbracket T\rrbracket$. The algebraic structure of the Selmer group (as a $\Lambda$-module) is encoded by certain invariants which have been extensively studied.
First consider the case when $E$ has good ordinary reduction at $p$.
By the $p$-adic Weierstrass Preparation Theorem, the characteristic ideal of the Pontryagin dual of the Selmer group is generated by a unique element $f_E^{(p)}(T)$, that can be expressed as a power of $p$ times a distinguished polynomial.
The $\mu$-invariant is the power of $p$ dividing $f_E^{(p)}(T)$ and the $\lambda$-invariant is its degree.
R. Greenberg has conjectured that when the residual representation on the $p$-torsion subgroup of $E$ is irreducible, then the $\mu$-invariant of the Selmer group vanishes, see \cite[Conjecture 1.11]{Gre99}.
Further, if the $p$-primary part of the Tate-Shafarevich group, denoted by $\Sha(E/\Q)[p^\infty]$, is finite, then one can show that the $\lambda$-invariant is at least as large as the Mordell-Weil rank of $E$ (see Lemma $\ref{lemma32}$).
However, this $\lambda$ may indeed be strictly larger than the rank, and one of our main objectives is to determine its behaviour on average .

When $E$ has supersingular reduction at $p$, the Selmer group is not $\Lambda$-cotorsion. 
This makes the analysis of the algebraic structure of the torsion part of the Selmer group particularly difficult.
Instead, we consider the plus and minus Selmer groups introduced by S. Kobayashi in \cite{Kob03}, which are known to be $\Lambda$-cotorsion.
The Iwasawa invariants $\mu^{+}$ and $\lambda^{+}$ (resp. $\mu^{-}$ and $\lambda^{-}$) of the plus (resp. minus) Selmer group are defined in an analogous manner.
In the supersingular case as well, there is much computational evidence towards the conjecture that the $\mu$-invariants $\mu^{+}$ and $\mu^-$ vanish.
Once again, under standard hypotheses on the Tate-Shafarevich group, both $\lambda^+$ and $\lambda^-$ are known to be greater than or equal to the Mordell-Weil rank of $E$ (see Lemma $\ref{lemma32}$).

\par The main goal of this article is prove results about the variation of the Iwasawa invariants as the pair $(E,p)$ varies such that $E$ has good reduction (ordinary or supersingular) at $p$.
More precisely, we analyze the following two separate but interrelated problems.
\begin{enumerate}
 \item 
For a fixed elliptic curve $E$, how do the Iwasawa invariants vary as $p$ varies over all odd primes $p$ at which $E$ has good reduction?
 \item 
 For a fixed prime $p$, how do the Iwasawa invariants vary as $E$ varies over all elliptic curves (with good reduction at $p$)?
\end{enumerate}

Greenberg studied the first question when $E_{/\Q}$ has rank zero and $p$ varies over the primes of good ordinary reduction (see \cite[Theorems 4.1 and 5.1]{Gre99}).
In Theorem \ref{theorem2}, we generalize this result to include the case of supersingular primes.
We show that a conjecture of J. Coates and R. Sujatha on the vanishing of $\mu$-invariants of fine Selmer groups holds for density one primes (see Corollary \ref{corollary: conjecture A}).
Under natural assumptions, we prove similar results for higher rank elliptic curves (see Theorems \ref{theorem3} and \ref{theorem4}).
The results in both ordinary and supersingular cases lead us to make the following conjecture (which is proved for elliptic curves of rank zero).
\begin{lconj}
Let $E_{/\Q}$ be an elliptic curve of rank $r_E$.
For $100\%$ of the primes $p$ at which $E$ has good ordinary reduction (resp. supersingular), $\mu=0$ and $\lambda=r_E$ (resp. $\mu^+=\mu^-=0$ and $\lambda^{+}=\lambda^{-}=r_E$).
\end{lconj}

The second question is at the intersection of arithmetic statistics and Iwasawa theory.
The area of arithmetic statistics concerns the behaviour of number theoretic objects in families, and offers a probabilistic model that seeks to explain numerous phenomena in the statistical behaviour of Selmer groups.
The investigations in this paper show that there is promise in the analysis of the average behaviour of Iwasawa invariants.
The main results we prove are Theorems $\ref{theorem6}$, $\ref{supersingulartheorem}$ and $\ref{main result rank 0 good reduction}$.
Our results indicate that it is reasonable to expect that for a fixed prime $p$, as we vary over all rank 0 elliptic curves over $\Q$ with good ordinary (resp. supersingular) reduction at $p$ ordered by height, a positive proportion of them have trivial $p$-primary Selmer group (considered over the cyclotomic $\Z_p$-extension of $\Q$).
In fact, the results suggest that the proportion of elliptic curves of rank $0$ with trivial $p$-primary Selmer group approaches $100\%$ as $p\rightarrow \infty$ (see Conjecture $\ref{lastconj}$).

\section{Background and Preliminaries}

\subsection{} Let $\Gamma:=\Gal(\Q^{\cyc}/\Q)\simeq \Z_p$.
The Iwasawa algebra $\Lambda$ is the completed group algebra $\Z_p\llbracket \Gamma \rrbracket :=\varprojlim_n \Z_p[\Gamma/\Gamma^{p^n}]$.
After fixing a topological generator $\gamma$ of $\Gamma$, there is an isomorphism of rings $\Lambda\cong\Z_p\llbracket T\rrbracket $, by sending $\gamma -1$ to the formal variable $T$.

\par Let M be a cofinitely generated cotorsion $\Lambda$-module.
The \emph{Structure Theorem of $\Lambda$-modules} asserts that the Pontryagin dual of M, denoted by $\rm{M}^{\vee}$, is pseudo-isomorphic to a finite direct sum of cyclic $\Lambda$-modules.
In other words, there is a map of $\Lambda$-modules
\[
\textrm{M}^{\vee}\longrightarrow  \left(\bigoplus_{i=1}^s \Lambda/(p^{m_i})\right)\oplus \left(\bigoplus_{j=1}^t \Lambda/(h_j(T)) \right)
\]
with finite kernel and cokernel.
Here, $m_i>0$ and $h_j(T)$ is a distinguished polynomial (i.e. a monic polynomial with non-leading coefficients divisible by $p$).
The characteristic ideal of $\rm{M}^\vee$ is (up to a unit) generated by the characteristic element,
\[
f_{\rm{M}}^{(p)}(T) := p^{\sum_{i} m_i} \prod_j h_j(T).
\]
The $\mu$-invariant of M is defined as the power of $p$ in $f_{\rm{M}}^{(p)}(T)$.
More precisely,
\[
\mu_p(\textrm{M}):=\begin{cases}0 & \textrm{ if } s=0\\
\sum_{i=1}^s m_i & \textrm{ if } s>0.
\end{cases}
\]
The $\lambda$-invariant of M is the degree of the characteristic element, i.e.
\[
\lambda_p(\textrm{M}) := \sum_{j=1}^t \deg h_j(T).
\]

\subsection{} Let $E$ be an elliptic curve over $\Q$ with good reduction at $p$. It shall be assumed throughout that the prime $p$ is odd.
Let $N$ denote the conductor of $E$ and set $S$ to denote the set of primes which divide $Np$. Let $\Q_S$ be the maximal algebraic extension of $\Q$ which is unramified at the primes $v\notin S$. Set $E[p^\infty]$ to be the Galois module of all $p$-power torsion points in $E(\closure{\Q})$.

\par First, consider the case when $E$ has good (ordinary or supersingular) reduction at $p$.
Let $v$ be a prime in $S$.
For any finite extension $L/\Q$ contained in $\Q^{\cyc}$, write
\[
J_v(E/L) = \bigoplus_{w|v} H^1\left( L_w, E\right)[p^\infty]
\]
where the direct sum is over all primes $w$ of $L$ lying above $v$.
Then, the \emph{$p$-primary Selmer group over $\Q$} is defined as follows
\[
\Sel_{p^\infty}(E/\Q):=\ker\left\{ H^1\left(\Q_S/\Q,E[p^{\infty}]\right)\longrightarrow \bigoplus_{v\in S} J_v(E/\Q)\right\}.
\]
This Selmer group fits into a short exact sequence 
\begin{equation}
\label{sesSelmer}
0\rightarrow E(\Q)\otimes \Q_p/\Z_p\rightarrow \Sel_{p^{\infty}}(E/\Q)\rightarrow \Sha(E/\Q)[p^{\infty}]\rightarrow 0,
\end{equation}
see \cite{CS00book}.
Here, $\Sha(E/\Q)$ is the \emph{Tate-Shafarevich group}
\[\Sha(E/\Q):=\left\{H^1(\closure{\Q}/\Q, E[p^{\infty}])\rightarrow \prod_{l} H^1(\closure{\Q}_l/\Q_l, E[p^{\infty}])\right\}.\]
Now, define
\[
J_v(E/\Q^{\cyc}) = \varinjlim J_v(E/L)
\]
where $L$ ranges over the number fields contained in $\Q^{\cyc}$ and the inductive limit is taken with respect to the restriction maps.
Taking direct limits, the \emph{$p$-primary Selmer group over $\Q^{\cyc}$} is defined as follows
\[
\Selp:=\ker\left\{ H^1\left(\Q_S/\Q^{\cyc},E[p^{\infty}]\right)\longrightarrow \bigoplus_{v\in S} J_v(E/\Q^{\cyc})\right\}.
\]

As mentioned in the introduction, when $p$ is a prime of good supersingular reduction, the Pontryagin dual of $\Selp$ is not $\Lambda$-torsion.
In this case, one studies the plus and minus Selmer groups which we describe below.

Let $E_{/\Q}$ be an elliptic curve with supersingular reduction at $p$.
Denote by $\Q_n^{\cyc}$ the $n$-th layer in the cyclotomic $\Z_p$-extension, with $\Q_0^{\cyc}:=\Q$.
Set $\widehat{E}$ to be the formal group of $E$ over $\Z_p$.
Let $L$ be a finite extension of $\Q_p$ with valuation ring $\mathcal{O}_L$, let $\widehat{E}(L)$ denote $\widehat{E}(\mathfrak{m}_L)$, where $\mathfrak{m}_L$ is the maximal ideal in $L$.
Write $\fp$ for the (unique) prime above $p$ in $\Q^{\cyc}$, and for the prime above $p$ in every finite layer of the cyclotomic tower.
Define the plus and minus norm groups as follows
\[ \widehat{E}^+(\Q_{n,\fp}^{\cyc}) :=
\left\{P\in \widehat{E}(\Q_{n,\fp}^{\cyc}) \mid \tr_{n/m+1} (P)\in \widehat{E}(\Q_{m,\fp}^{\cyc}), \text{ for }0\leq m < n\text{ and }m \text{ even }\right\},\]
\[\widehat{E}^-(\Q_{n,\fp}^{\cyc}):=\left\{P\in \widehat{E}(\Q_{n,\fp}^{\cyc}) \mid \tr_{n/m+1} (P)\in \widehat{E}(\Q_{m,\fp}^{\cyc}),\text{ for }0\leq m < n\text{ and }m 
\text{ odd }\right\},\]
where $\tr_{n/m+1}:\widehat{E}(\Q_{n,\fp}^{\cyc})\rightarrow \widehat{E}(\Q_{m+1,\fp}^{\cyc})$ denotes the trace map with respect to the formal group law on $\widehat{E}$. The completion $\Q^{\cyc}_{\fp}$ is the union of completions $\bigcup_{n\geq 1} \Q^{\cyc}_{n,\fp}$, and set $\widehat{E}^{\pm}(\Q_{\fp}^{\cyc}):=\bigcup_{n\geq 1}\widehat{E}^{\pm}(\Q_{n,\fp}^{\cyc})$. 
Define
\[J_p^{\pm}(E/\Q^{\cyc}):=\frac{H^1(\Q_{\fp}^{\cyc},E[p^{\infty}])}{\widehat{E}^{\pm}(\Q_{\fp}^{\cyc})\otimes \Q_p/\Z_p},\] where the inclusion
\[\widehat{E}^{\pm}(\Q_{\fp}^{\cyc})\otimes \Q_p/\Z_p\hookrightarrow H^1(\Q_{\fp}^{\cyc},E[p^{\infty}])\] is induced via the Kummer map.
For $v\in S\setminus \{p\}$, set $J_v^{\pm}(E/\Q^{\cyc})$ to be equal to $J_v(E/\Q^{\cyc})$.
The \emph{plus and minus Selmer groups} are defined as follows
\[\Sel_{p^\infty}^{\pm}(E/\Q^{\cyc}):=\ker\left\{H^1\left(\Q_S/\Q^{\cyc},E[p^{\infty}]\right)\rightarrow \bigoplus_{v\in S} J_v^{\pm}(E/\Q^{\cyc})\right\}.\]
For each choice of sign $\ddag\in \{+,-\}$, set $f_E^{(p),\ddag}(T)$ to denote the characteristic element of $\Sel_{p^\infty}^{\ddag}(E/\Q^{\cyc})^{\vee}$, with $\mu_{p}^{\ddag}(E)$, $\lambda_{p}^{\ddag}(E)$ defined analogously.

Let $E$ be an elliptic curve with good (ordinary or supersingular) reduction at $p\geq 5$.
In order to state results in both ordinary and supersingular case at once, we set for the remainder of the article the following notation,
\[
\Selast := 
\begin{cases}
\Selpm & \textrm{if } E \textrm{ has supersingular reduction at } p, \textrm{ where } \ddag=\pm,\\
\Selp & \textrm{if } E \textrm{ has ordinary reduction at }p.
\end{cases}
\]

\subsection{} Next, we introduce the fine Selmer group.
Let $E$ be an elliptic curve over $\Q$ and $p$ be any odd prime. At each prime $v\in S$, set \[\mathcal{K}_v(E/\Q^{\cyc}) :=\bigoplus_{\eta|v} H^1(\Q^{\cyc}_{\eta}, E[p^{\infty}]).\]
The \emph{$p$-primary fine Selmer group} of $E$ is defined as follows
\[\Sel^0_{p^{\infty}}(E/\Q^{\cyc}) := \ker\left\{ H^1(\Q_S/\Q^{\cyc}, E[p^{\infty}])\longrightarrow \bigoplus_{v\in S} \mathcal{K}_v(E/ \Q^{\cyc}) \right\}.\]
By the result of Kato mentioned in the introduction, the Pontryagin dual of $\Sel^0_{p^{\infty}}(E/\Q^{\cyc})$ is known to be a cotorsion $\Lambda$-module independent of the reduction type at $p$.
Further, it is conjectured that this fine Selmer group is a cotorsion $\Z_p$-module (i.e. the corresponding $\mu$-invariant vanishes), see \cite[Conjecture A]{CS05}.

\subsection{} In what follows, M will be a cofinitely generated cotorsion $\Lambda$-module. Note that $H^i(\Gamma, \rm{M})$ is always zero for $i\geq 2$.

\begin{Lemma}\label{balancedrank}
Let M be a cofinitely generated cotorsion $\Lambda$-module.
Then, 
\[\corank_{\Z_p} \rm{M}^{\Gamma}=\corank_{\Z_p} \rm{M}_{\Gamma}.\]
\end{Lemma}
\begin{proof}
Note that $H^1(\Gamma, \rm{M})$ may be identified with $\rm{M}_{\Gamma}$ (see \cite[Proposition 1.7.7]{NSW08}).
It follows from \cite[Theorem 1.1]{How02} that 
\[
\corank_{\Lambda} \rm{M}= \corank_{\Z_p}H^0(\Gamma, \rm{M})-\corank_{\Z_p}H^1(\Gamma, \rm{M}).
\]
Since $\rm{M}$ is assumed to be cotorsion over $\Lambda$, the result follows.
\end{proof}

When the cohomology groups $H^0(\Gamma, \rm{M})$ and $H^1(\Gamma, \rm{M})$ are finite, the (classical) \emph{Euler characteristic} $\chi(\Gamma, \rm{M})$ is defined as the alternating product
\[\chi(\Gamma, \rm{M})=\prod_{i\geq 0} \left(\# H^i(\Gamma, \rm{M})\right)^{(-1)^i}.\]
On the other hand, when the cohomology groups $H^i(\Gamma, \rm{M})$ are not finite, there is a generalized version denoted by $\chi_t(\Gamma, \rm{M})$.
Since $H^1(\Gamma, \rm{M})$ is isomorphic to the group of coinvariants $H_0(\Gamma, \rm{M})=\rm{M}_{\Gamma}$, there is a natural map 
\[
\Phi_{\rm{M}}:\rm{M}^{\Gamma}\rightarrow \rm{M}_{\Gamma}
\]
sending $x\in \rm{M}^{\Gamma}$ to the residue class of $x$ in $\rm{M}_{\Gamma}$.
We say that the \emph{truncated Euler characteristic} is defined if the kernel and cokernel of $\Phi_{\rm{M}}$ are finite.
In this case, the $\chi_t(\Gamma, \rm{M})$ is defined to be the following quotient,
\[
\chi_t(\Gamma, \rm{M}):=\frac{\#\op{cok}(\Phi_{\rm{M}})}{\#\ker(\Phi_{\rm{M}})}.
\]
It is easy to check that when $\chi(\Gamma, \rm{M})$ is defined, so is $\chi_t(\Gamma, \rm{M})$. 
In fact,
\[\chi_t(\Gamma, \rm{M})=\chi(\Gamma, \rm{M}).\]
Express the characteristic element $f_{\rm{M}}^{(p)}(T)$ as a polynomial, 
\[
f_{\rm{M}}^{(p)}(T)=c_0+c_1T+\dots +c_d T^d.
\]
Let $r_{\rm{M}}$ denote the order of vanishing of $f_{\rm{M}}^{(p)}(T)$ at $T=0$.
For $a,b\in \Q_p$, we write $a\sim b$ if there is a unit $u\in \Z_p^{\times}$ such that $a=bu$.
\begin{Lemma}[S. Zerbes]
\label{lemmazerbes}
Let M be a cofinitely generated cotorsion $\Lambda$-module. Assume that the kernel and cokernel of $\Phi_{\rm{M}}$ are finite.
Then,
\begin{enumerate}
\item ${r_{\rm{M}}}=\corank_{\Z_p}(\rm{M}^{\Gamma})=\corank_{\Z_p}(\rm{M}_{\Gamma})$.
\item $c_{r_{\rm{M}}}\neq 0$.
\item $c_{r_{\rm{M}}}\sim \chi_t(\Gamma, \rm{M})$. 
\end{enumerate}
\end{Lemma}
\begin{proof}
See \cite[Lemma 2.11]{zerbes09}.
\end{proof}

In particular, the classical Euler characteristic $\chi(\Gamma, \rm{M})$ is defined if and only if $r_{\rm{M}}=0$. 
When this happens, the constant coefficient $c_0\sim \chi(\Gamma, \rm{M})$.

We specialize the discussion on Euler characteristics to Selmer groups of elliptic curves.
Let $\ddag$ be a choice of sign and recall the definition of $\Selast$.
When $E$ has good ordinary reduction at $p$, the choice of $\ddag$ is irrelevant.
Denote by $\chi_t^{\ddag}(\Gamma, E[p^{\infty}])$ (resp. $\chi^{\ddag}(\Gamma, E[p^{\infty}])$) the truncated (resp. classical) Euler characteristic of the Selmer group $\Selast$. The invariants $\mu_p^{\ddag}(E)$ and $\lambda_p^{\ddag}(E)$ simply refer to $\mup$ and $\lambdap$ when $E$ has good ordinary reduction at $p$.
When $E$ has good ordinary reduction at $p$, we shall drop the sign $\ddag$ from the notation.
It follows from Lemma $\ref{lemmazerbes}$ that the truncated Euler characteristic is always an integer.

\section{Results for a fixed elliptic curve and varying prime}
In this section, we study the variation of the classical and the truncated Euler characteristic as $p$ varies.
Fix a pair $(E,p)$ such that
\begin{enumerate}
\item $p$ is odd.
 \item $E$ is defined over $\Q$ and has good reduction at $p$.
\end{enumerate}
We record some lemmas which are required throughout this section.
Recall that when $E$ has good ordinary reduction, $\Selast$ simply refers to $\Selp$.
\begin{Th}\label{control}
Let $\ddag\in \{+,-\}$ be a choice of sign.
Let $E_{/\Q}$ be an elliptic curve with good reduction at $p>2$.
Then, there is a natural map
\[
\Sel_{p^\infty}(E/\Q)\rightarrow \Sel_{p^\infty}^{\ddag}(E/\Q^{\cyc})^{\Gamma}
\]
with finite kernel and cokernel.
\end{Th}
\begin{proof}
When $E$ has good ordinary reduction at $p$, this follows from Mazur's control theorem, see \cite{mazur72} or \cite[Theorem 1.2]{Gre99}.
In the good supersingular reduction case, the result follows from the proof of \cite[Lemma 3.9]{Kim13} (see also \cite[Proposition 5.1]{ray2}).
\end{proof}
The following lemma gives a criterion for when the classical Euler characteristic $\chi^{\ddag}(\Gamma, E[p^{\infty}])$ is well-defined.
\begin{Lemma}\label{lemma31}
Let $\ddag\in \{+,-\}$ be a choice of sign.
Let $E_{/\Q}$ be an elliptic curve with good reduction at $p>2$.
The following are equivalent.
\begin{enumerate}
 \item\label{31c1} The classical Euler characteristic $\chi^{\ddag}(\Gamma, E[p^{\infty}])$ is well-defined.
 \item\label{31c2} $\Selast^{\Gamma}$ is finite.
 \item\label{31c3} The Selmer group $\Sel_{p^{\infty}}(E/\Q)$ is finite.
 \item\label{31c4} The Mordell-Weil group $E(\Q)$ is finite, i.e. the Mordell-Weil rank is 0.
\end{enumerate} 
\end{Lemma}

\begin{proof}
Recall that the Selmer group $\Selast$ is a cotorsion $\Lambda$-module.
By Lemma $\ref{balancedrank}$, $\Selast^{\Gamma}$ is finite if and only if $\Selast_{\Gamma}$ is finite.
This shows that $\eqref{31c1}$ and $\eqref{31c2}$ are equivalent.

Theorem $\ref{control}$ asserts that there is a natural map
\[
\Sel_{p^\infty}(E/\Q)\rightarrow \Selast^{\Gamma},
\]
with finite kernel and cokernel.
Hence, the conditions $\eqref{31c2}$ and $\eqref{31c3}$ are equivalent.
By the work of V. Kolyvagin, $\Sha(E/\Q)$ is known to be finite when $E(\Q)$ is finite (see \cite{Kol89}).
Thus, it follows from \eqref{sesSelmer} that conditions $\eqref{31c3}$ and $\eqref{31c4}$ are equivalent.
\end{proof}

\begin{Lemma}\label{lemma32}
Let $\ddag\in \{+,-\}$ be a choice of sign.
Let $E_{/\Q}$ be an elliptic curve with good reduction at $p>2$.
Assume that $\chi_t^{\ddag}(\Gamma, E[p^{\infty}])$ is well-defined and $\Sha(E/\Q)[p^{\infty}]$ is finite.
Let $r_E$ denote the Mordell-Weil rank of $E(\Q)$.
Then, $r_E$ is equal to the order of vanishing of $f_{E}^{(p),\ddag}(T)$ at $T=0$.
In particular, we have that $\lambda_p^{\ddag}(E)\geq r_E$. 
\end{Lemma}
\begin{proof}
Since it is assumed that $\Sha(E/\Q)[p^{\infty}]$ is finite, it follows from $\eqref{sesSelmer}$ that
\[
r_E=\corank_{\Z_p} \Sel_{p^\infty}(E/\Q).
\]
By Lemma $\ref{lemmazerbes}$, it suffices to show that \[r_E=\corank_{\Z_p}\Selast^{\Gamma}.\] Therefore the result follows from Theorem $\ref{control}$.
\end{proof}

\begin{Lemma}
\label{TECmulambda}
Let M be a cofinitely generated and cotorsion $\Lambda$-module such that $\phi_{\rm{M}}$ has finite kernel and cokernel. 
Let $r_{\rm{M}}$ be the order of vanishing of $f_{\rm{M}}^{(p)}(T)$ at $T=0$.
Then, the following are equivalent.
\begin{enumerate}
\item\label{TECmulambdac1} $\chi_t(\Gamma, \rm{M})=1$,
\item\label{TECmulambdac2} $\mu(\rm{M})=0$ and $\lambda(\rm{M})=r_{\rm{M}}$.
\end{enumerate}
\end{Lemma}

\begin{proof}
Suppose that $\chi_t(\Gamma, \rm{M})=1$.
Write $f_{\rm{M}}^{(p)}(T)=T^{r_{\rm{M}}} g_{\rm{M}}^{(p)}(T)$ where $g_{\rm{M}}^{(p)}(T)\in \Lambda$ and $g_{\rm{M}}(0)\neq 0$.
By Lemma $\ref{lemmazerbes}$, \[\abs{g_{\rm{M}}^{(p)}(0)}_p^{-1}=\chi_t(\Gamma, \rm{M})=1.\]
In particular, $f_{\rm{M}}^{(p)}(T)$ and $g_{\rm{M}}^{(p)}(T)$ are distinguished polynomials.
Since $g_{\rm{M}}^{(p)}(0)$ is a unit, it follows that $g_{\rm{M}}^{(p)}(T)$ is a unit.
Since $g_{\rm{M}}^{(p)}(T)$ is a distinguished polynomial, it follows that
\[g_{\rm{M}}^{(p)}(T)=1\text{ and }f_{\rm{M}}^{(p)}(T)=T^{r_{\rm{M}}}.
\] 
Therefore, $\mu(\rm{M})=0$ and $\lambda(\textrm{M})=\deg f_M^{(p)}(T)=r_M$. 

Conversely, suppose that $\mu(\rm{M})=0$ and $\lambda(\textrm{M})=r_{\rm{M}}$.
Since $\mu(\rm{M})=0$, it follows that $f_{\rm{M}}^{(p)}(T)$ and $g_{\rm{M}}^{(p)}(T)$ are distinguished polynomials.
The degree of $f_{\rm{M}}^{(p)}(T)$ is $\lambda_{\rm{M}}=r_{\rm{M}}$. 
It follows that $g_{\rm{M}}^{(p)}(T)$ is a constant polynomial and hence, $g_{\rm{M}}^{(p)}(T)=1$.
By Lemma $\ref{lemmazerbes}$, \[\chi_t(\Gamma, \textrm{M})=\abs{ g_{\textrm{M}}(0)}_p^{-1}=1.\]
\end{proof}

\begin{Proposition}\label{prop34}
Let $\ddag\in \{+,-\}$ be a choice of sign.
Let $E_{/\Q}$ be an elliptic curve of Mordell-Weil rank $r_E$ with good reduction at $p>2$.
Suppose that these additional conditions hold.
\begin{enumerate}[(i)]
 \item\label{34c2} The truncated Euler characteristic $\chi_t^{\ddag}(\Gamma, E[p^{\infty}])$ is defined.
 \item\label{34c3} $\Sha(E/\Q)[p^{\infty}]$ is finite.
\end{enumerate}
Then the following statements are equivalent
\begin{enumerate}
 \item $\chi_t^{\ddag}(\Gamma, E[p^{\infty}])=1$,
 \item $\mu_p^{\ddag}(E)=0$ and $\lambda_p^{\ddag}(E)=r_E$.
\end{enumerate}
\end{Proposition}
\begin{proof}
Lemma $\ref{lemma32}$ asserts that $r_E$ is equal to the order of vanishing of $f_{E}^{(p),\ddag}(T)$ at $T=0$.
The assertion follows from Lemma $\ref{TECmulambda}$.
\end{proof}

\subsection{Elliptic curves over $\Q$ with rank zero}
In this subsection, we study the variation of the classical Euler characteristic as $p$ varies over primes of good reduction.
\begin{Corollary}\label{corollary35}
Let $E$ be an elliptic curve over $\Q$ with good reduction at $p>2$ for which the Mordell-Weil rank of $E$ is zero. 
The following are equivalent.
\begin{enumerate}
 \item\label{35p1} $\chi^{\ddag}(\Gamma, E[p^{\infty}])=1$,
 \item\label{35p2} $\mu_p^{\ddag}(E)=0$ and $\lambda_p^{\ddag}(E)=0$.
 \item\label{35p3} $\Selast$ is finite.
 \item\label{35p4} $\Selast=0$.
\end{enumerate}
\end{Corollary}
\begin{proof}
Since $E$ is assumed to have rank zero, the Tate-Shafarevich group $\Sha(E/\Q)$ is finite.
It follows from Lemma $\ref{lemma31}$ that the Euler characteristic $\chi^{\ddag}(\Gamma, E[p^{\infty}])$ is well-defined, and by Proposition $\ref{prop34}$ that $\eqref{35p1}$ and $\eqref{35p2}$ are equivalent.
From the Structure Theorem of finitely generated $\Lambda$-modules, it is clear that $\eqref{35p2}$ and $\eqref{35p3}$ are equivalent conditions.
It is known that $\Selast$ contains no proper finite index submodules (see \cite[Proposition 4.14]{Gre99} for the case when $p$ is ordinary and \cite[Theorem 3.14]{Kim13} for when $p$ is supersingular).
Hence, $\eqref{35p3}$ and $\eqref{35p4}$ are equivalent.
\end{proof}

Let $E_{/\Q}$ be an elliptic curve.
Denote by $S^{\bad}$ the finite set of primes at which $E$ has bad reduction.
Denote by $S^{\good}$ the primes at which $E$ has good reduction (i.e., either ordinary reduction or supersingular reduction); write $S^{\good}=S^{\ord}\cup S^{\op{ss}}$.
When $E$ has good reduction at a fixed prime $p$, we set $\widetilde{E}$ to denote the reduced curve over $\F_p$.

The following result was initially proved by Greenberg for good ordinary primes. 
\begin{Th}\label{theorem1}
Let $E$ be an elliptic curve over $\Q$ such that $E(\Q)$ is finite.
Let $\Sigma\subset S^{\ord}$ be the set of primes at which $p$ divides $\# \widetilde{E}(\F_p)$.
Let $\Sigma'\subset S^{\ord}$ be the finite set of primes $p$ such that either
\begin{enumerate}
\item $p=2$.
\item $p$ divides $\# \Sha(E/\Q)$.
\item $p$ divides the Tamagawa product $\prod_{l\in S^{\bad}} c_l(E)$.
\end{enumerate}
Then for all primes $p\in S^{\ord}\setminus (\Sigma \cup \Sigma')$, we have that $\Selp=0$.
\end{Th}

Observe that $\Sigma^\prime$ is always a finite set.
However, the set of primes $\Sigma$, referred to as the set of \emph{anomalous primes}, is known to be a set of Dirichlet density zero.
The Lang-Trotter Conjecture predicts that the proportion of anomalous primes $p<X$ is $C \cdot \frac{\sqrt{X}}{\log X}$ for some constant $C$.
It has recently been shown that all but a density zero set of elliptic curves have infinitely many anomalous primes \cite[Corollary 4.3]{BBKNT19}.
However, in special cases one can in fact show that $C=0$ \cite[\S 1.1]{Rid10}.
By the Hasse inequality, a prime $p\geq 7$ is anomalous for an elliptic curve defined over $\Q$ if and only if $a_p = p+ 1 - \#\widetilde{E}(\mathbb{F}_p) = 1$.
When $p\geq 7$ and an elliptic curve has 2-torsion, $a_p$ is even and hence $C=0$.
For more examples where the set $\Sigma$ is finite, and detailed discussion on this subject, we refer the reader to \cite{mazur72, Ols76, Qin16}.

\begin{proof}[Proof of Theorem $\ref{theorem1}$] Since the Mordell-Weil rank of $E$ is assumed to be zero, the Euler characteristic $\chi(\Gamma, E[p^{\infty}])$ is well-defined.
The following formula for the Euler characteristic is well known (see \cite[Theorem 3.3]{CS00book})
\begin{equation}\label{ecf1}
\chi(\Gamma, E[p^{\infty}])\sim \frac{\# \Sha(E/\Q)[p^{\infty}]\times \left(\prod_{l\in S^{\bad}}c_l(E)\right)\times \left(\# \widetilde{E}(\F_p)\right)^2}{\left(\# E(\Q)[p^\infty]\right)^2}.
\end{equation}
Since the Euler characteristic is an integer, the result follows immediately from Corollary $\ref{corollary35}$.
\end{proof}

We prove a similar result for primes of supersingular reduction.
\begin{Th}\label{theorem2}
Let $E$ be an elliptic curve over $\Q$ such that $E(\Q)$ is finite. Let $\Upsilon\subset S^{\op{ss}}$ be the finite set of primes $p$ such that either
\begin{enumerate}
\item $p=2$.
 \item $p$ divides $\# \Sha(E/\Q)$.
 \item $p$ divides the Tamagawa product $\prod_{l\in S^{\bad}} c_l(E)$.
\end{enumerate}
Then for all primes $p\in S^{\op{ss}}\setminus \Upsilon$, we have that \[\Sel_{p^{\infty}}^+(E/\Q^{\cyc})=0 \text{ and }\Sel_{p^{\infty}}^-(E/\Q^{\cyc})=0.\]
\end{Th}
\begin{proof}
Since $E(\Q)$ is assumed to be finite, the Euler characteristic $\chi^{\pm}(\Gamma, E[p^{\infty}])$ is well-defined.
Since $E$ has supersingular reduction at $p$, $E[p]$ is irreducible as a $\op{G}_{\Q}$-module; hence, $E(\Q)[p]$ is the trivial group.
Finiteness of $E(\Q)$ implies that $E(\Q)\otimes \Q_p/\Z_p=0$. 
Therefore
\[\# \Sel_{p^{\infty}}(E/\Q)=\# \Sha(E/\Q)[p^{\infty}].\]
By \cite[Theorem 5.15]{leisujatha}, we know that
\begin{equation}\label{ecf2}\chi^{\pm}(\Gamma, E[p^{\infty}])=\#\Sha(E/\Q)[p^{\infty}]\times \left(\prod_{l\in S^{\bad}}c_l(E)\right).\end{equation}
The result follows from Corollary $\ref{corollary35}$.
\end{proof}

\begin{Remark}
\label{remark: trivial rank every layer}
It is clear that the $\mu$ and $\lambda$ invariants do not change for each of the layers of the cyclotomic $\mathbb{Z}_p$-extension.
Let $E_{/\Q}$ be a rank 0 elliptic curve and $p\in S^{\ord}\setminus (\Sigma \cup \Sigma^{\prime})$ or $p\in S^{\op{ss}}\setminus \Upsilon$. 
Since $\lambda_p^{\ddag}(E) \geq r_E$, the Mordell-Weil rank of the elliptic curve must remain zero at each layer in the cyclotomic $\mathbb{Z}_p$-extension.
\end{Remark}
The following result gives evidence for the conjecture of Coates and Sujatha on the vanishing of the $\mu$-invariant of fine Selmer groups.
In fact, more is true.

\begin{Corollary}
\label{corollary: conjecture A}
Let $E$ be a rank 0 elliptic curve defined over $\Q$.
Then, the $p$-primary fine Selmer group $\Sel^0_{p^\infty}(E/\Q^{\cyc})$ is trivial for density one primes.
\end{Corollary}
\begin{proof}
Given any elliptic curve, it has good reduction at all but finitely many primes.
When $p$ is a prime of good (ordinary or supersingular) reduction, note that the $p$-primary fine Selmer group is a subgroup of $\Sel^{\ddag}_{p^\infty}(E/\Q^{\cyc})$.
The result is immediate from Theorems \ref{theorem1} and \ref{theorem2}.
\end{proof}


\subsection{Elliptic curves over $\Q$ with positive rank}
Next, we consider the case when the elliptic curve $E$ has positive rank. 

When $E$ has good ordinary reduction at $p$, there is a notion of the $p$-adic height pairing, which is the $p$-adic analog of the usual height pairing, and was studied extensively in \cite{schneider1, schneider2}.
This pairing is conjectured to be non-degenerate, and the $p$-adic regulator $R_p(E/\Q)$ is defined to be the determinant of this pairing.
\begin{Th}\label{pbsdconj}[{B. Perrin-Riou \cite{perrin94}, P. Schneider \cite[Theorem 2']{schneider2}}]
Let $E_{/\Q}$ be an elliptic curve of Mordell-Weil rank $r_E$ with good ordinary reduction at $p$.
Assume that
\begin{enumerate}
 \item $\Sha(E/\Q)[p^{\infty}]$ is finite.
 \item the $p$-adic height pairing is non-degenerate, i.e., the $p$-adic regulator is non-zero.
\end{enumerate}
Then, the order of vanishing of $f_E^{(p)}(T)$ at $T=0$ is equal to $r_E:=\op{rank} E(\Q)$. Express $f_E^{(p)}(T)$ as $T^{r_{E}}g_E^{(p)}(T)$. Then, $g_E^{(p)}(0)\neq 0$ and 
\[g_E^{(p)}(0)\sim \frac{\left(\frac{R_p(E/\Q)}{p^{r_E}}\right)\times \# \Sha(E/\Q)[p^{\infty}]\times \left(\prod_{l\in S^{\bad}}c_l(E)\right)\times \left(\# \widetilde{E}(\F_p)\right)^2}{\left( \# E(\Q)[p^\infty]\right)^2}.\]
\end{Th}
The above theorem implies that the right hand side of the equation is an integer since the left hand side is.
This result does not assume that the truncated Euler characteristic is well-defined.
The statement is proven in greater generality, in fact, for abelian varieties over number fields.

\begin{Corollary}
Let $E_{/\Q}$ be an elliptic curve of Mordell-Weil rank $r_E$ with good ordinary reduction at $p$.
Assume that
\begin{enumerate}
 \item $\Sha(E/\Q)[p^{\infty}]$ is finite.
 \item the $p$-adic Height pairing is non-degenerate.
 \item the truncated Euler characteristic $\chi_t(\Gamma,E[p^{\infty}])$ is defined.
\end{enumerate}
Then, we have that
\[
\chi_t(\Gamma, E[p^{\infty}])\sim \frac{\left(\frac{R_p(E/\Q)}{p^{r_E}}\right)\times \# \Sha(E/\Q)[p^{\infty}]\times \left(\prod_{l\in S^{\bad}}c_l(E)\right)\times \left(\# \widetilde{E}(\F_p)\right)^2}{\left( \# E(\Q)[p^\infty]\right)^2}.
\]
\end{Corollary}
\begin{proof}
The assertion a direct consequence of Theorem $\ref{pbsdconj}$ and Lemma $\ref{lemmazerbes}$.
\end{proof}

To analyze Iwasawa invariants on average, we shall apply Theorem $\ref{pbsdconj}$.
Let $E_{/\Q}$ be an elliptic curve.
As before, we denote by $S^{\ord}$ the set of primes $p$ at which $E$ has good ordinary reduction. Let $\Sigma$ be the set of anomalous primes and $\Sigma'$ the finite set of primes for which \begin{enumerate}
\item $p=2$,
 \item $p$ divides $\# \Sha(E/\Q)$.
 \item $p$ divides the Tamagawa product $\prod_{l\in S^{\bad}} c_l(E)$.
\end{enumerate}
Denote by $v_p$ the $p$-adic valuation on $\Q_p$ normalized by $v_p(p)=1$. Let $\Pi\subset S^{\ord}$ be the set of primes $p$ at which $v_p(R_p(E/\Q))\geq r_E$. In other words, it is the set of primes for which $p$ divides $\left(\frac{R_p(E/\Q)}{p^{r_E}}\right)$.
When $r_E=0$, the $p$-adic regulator $R_p(E/\Q)=1$, therefore the set $\Pi$ is empty.
On the other hand, the set of primes $\Pi$ need not be empty when $r_E\geq 1$.
The following result generalizes Theorem $\ref{theorem1}$ and is an easy consequence of Theorem $\ref{pbsdconj}$.
\begin{Th}\label{theorem3}
Let $E_{/\Q}$ be an elliptic curve such that $r_E\geq 1$.
Then for all primes $p\in S^{\ord}\setminus (\Sigma \cup \Sigma'\cup \Pi)$, we have that $f_E^{(p)}(T)=T^{r_E}$.
In particular, $\mup=0$ and $\lambdap=r_E$.
\end{Th}
\begin{proof}
Express the characteristic element $f_E^{(p)}(T)$ as $T^{r_{E}}g_E^{(p)}(T)$. Theorem $\ref{pbsdconj}$ asserts that $g_E^{(p)}(0)\neq 0$ and 
\[g_E^{(p)}(0)\sim \frac{\left(\frac{R_p(E/\Q)}{p^{r_E}}\right)\times \# \Sha(E/\Q)[p^{\infty}]\times \left(\prod_{l\in S^{\bad}}c_l(E)\right)\times \left(\# \widetilde{E}(\F_p)\right)^2}{\left( \# E(\Q)[p^\infty]\right)^2}.\]
Thus, $g_E^{(p)}(0)$ is a $p$-adic unit for $p\in S^{\ord}\setminus (\Sigma \cup \Sigma'\cup \Pi)$.
Therefore, $g_E^{(p)}(T)$ is a unit in $\Lambda$.
Since $f_E^{(p)}(T)$ is a product of a distinguished polynomial with a power of $p$, it follows that $g_E^{(p)}(T)=1$ and $f_E^{(p)}(T)=T^{r_E}$.
\end{proof}

While it is known that the set $\Sigma'$ is finite and $\Sigma$ is a density zero set of primes, the same is not known for $\Pi$.
For $N\geq 1$, let $\Pi^{\leq N}$ be the set of primes $p\in \Pi$ for which $5\leq p\leq N$.
Calculations in \cite{CM94} suggest that the set $\Pi$ is likely to have Dirichlet density zero, at least for elliptic curves of rank $1$ (see also \cite{Wut04}).
We compute $\Pi^{\leq 1000}$ for the first 10 elliptic curves $E/\Q$ of rank $2$ (ordered by conductor).
The calculations in the following table were done on sage.
\vspace{1.0cm}
\begin{center}
\begin{tabular}{c|c|c |c|c|c} 
 & Cremona Label & $\Pi^{\leq 1000}$ & & Cremona Label & $\Pi^{\leq 1000}$\\ [1 ex]
 \hline
 $1.$ & $389a$ & $\emptyset$ & $6.$ & $643a$ & $\emptyset$\\
 $2.$ & $433a$ & $\{13\}$& $7.$ & $655a$ & $\{7,31\}$\\
 $3.$ & $446d$ & $\{7\}$&$8.$ & $664a$ & $\{59\}$\\
 $4.$ & $563a$ & $\emptyset$&$9.$ & $681c$ & $\emptyset$ \\
 $5.$ & $571b$ & $\emptyset$&$10.$ & $707a$& \{29\}.\\
\end{tabular}
\end{center}
\vspace{1cm}

\begin{Remark}
For elliptic curves $E_{/\Q}$ with $r_E=1$, the above theorem asserts that for primes outside a set (possibly of Dirichlet density zero), the characteristic element $f_E^{(p)}(T) =T$.
Under the hypotheses that the $p$-adic height pairing on $E(\Q_{n}^{\cyc})$ is non-degenerate for all $n$ and that $\# \Sha(E/\Q_n^{\cyc})[p^\infty]$ is finite for all $n$, we know that $f_{\Sel^0}^{(p)}(T) =1$ (see \cite[pp. 104-105]{Wut04_thesis}).
Thus, in this case, the fine Selmer group $\Sel^0_{p^\infty}(E/\Q^{\cyc})$ is not only cofinitely generated as a $\Z_p$-module, but in fact it is finite.
On the other hand, when $r_E>1$ even though the fine Selmer group should be cofinitely generated as a $\Z_p$-module, it is not expected to be finite. 
\end{Remark}

Next, we consider the supersingular case.
Let $E_{/\Q}$ be an elliptic curve with good supersingular reduction at $p\geq 3$.
Note that $p$ divides $a_p:=p+1-\#\widetilde{E}(\F_p)$.
Hasse's bound states that $\abs{a_p}<2\sqrt{p}$; hence if $p\geq 5$, it forces that $a_p=0$.
However, when $p=3$, it is indeed possible for $a_p\neq 0$.
For simplicity, we shall assume that $a_p=0$.
In this setting, the Main conjectures were formulated in \cite{Kob03}.
When $a_p=0$, the Main conjecture has been proved in a preprint of X. Wan (see \cite{Wan14}).
Let $\ddag\in \{+,-\}$ be a choice of sign.
Perrin Riou \cite{PR93} formulated a $p$-adic L-function in the supersingular case, which is closely related to the plus and minus $p$-adic $L$-function defined by R. Pollack.
The $p$-adic Birch and Swinnerton-Dyer conjecture formulated by D. Bernardi and Perrin-Riou in \cite{BPR} predicts a formula for the leading term of the $p$-adic L-function.
This conjecture is reformulated in terms of Pollack's $p$-adic $L$-functions in \cite{sprung15}.

Let $\log_p$ be a branch of the $p$-adic logarithm, $\chi$ the $p$-adic cyclotomic character, and $r_E$ be the Mordell-Weil rank of $E$.
Let $\ddag\in \{+,-\}$, and denote by $R^{\ddag}_p(E/\Q)$ the signed $p$-adic regulator (defined up to $p$-adic unit).
The convention in \textit{loc. cit.} is to choose a generator $\gamma$ of the cyclotomic $\Z_p$ extension, and divide the regulator (defined w.r.t. the choice of $\gamma$) by $\log_p(\chi(\gamma))^{r_E}$.
We are however, not interested in the exact value of the regulator, but only the value up to a $p$-adic unit. 
Therefore, we simply work with the fraction $\left(\frac{R^{\ddag}_p(E/\Q)}{p^{r_E}}\right)$.
Lemma $\ref{lemma32}$ asserts that the order of vanishing of $f_E^{(p), \ddag}(T)$ at $T=0$ is equal to $r_E$.
Express $f_E^{(p), \ddag}(T)$ as a product $T^{r_E}g_E^{(p), \ddag}(T)$.
The following conjecture is equivalent to the $p$-adic Birch and Swinnerton-Dyer conjecture.
\begin{Conjecture}\label{supersingularbsd}
Let $E$ be an elliptic curve with good supersingular reduction at the prime $p$ and $\ddag\in \{+,-\}$.
Then,  
\[g_E^{(p), \ddag}(0)\sim \left(\frac{R_p^{\ddag}(E/\Q)}{p^{r_E}}\right)\times \# \Sha(E/\Q)[p^{\infty}]\times \left(\prod_{l\in S^{\bad}}c_l(E)\right).\]
\end{Conjecture}
Let $\Pi^{\ddag}\subset S^{\op{ss}}$ be the set of primes $p$ at which $v_p(R_p^{\ddag}(E/\Q))\geq r_E$.
In other words, it is the set of primes for which $p$ divides $\left(\frac{R_p^{\ddag}(E/\Q)}{p^{r_E}}\right)$.
The following result is proved using the same strategy as that of Theorem $\ref{theorem3}$, so we skip the proof.
\begin{Th}\label{theorem4}
Assume that Conjecture $\ref{supersingularbsd}$ holds.
Let $E_{/\Q}$ be an elliptic curve with Mordell-Weil rank $r_E\geq 1$.
Let $p$ be a prime at which $E$ has good supersingular reduction.
Then for all primes $p\in S^{\op{ss}}\setminus (\Sigma \cup \Sigma'\cup \Pi^{\ddag})$, we have that $f_E^{(p),\ddag}(T)=T^{r_E}$. In particular, $\mu_p^{\ddag}(E)=0$ and $\lambda_p^{\ddag}(E)=r_E$.
\end{Th}

It seems reasonable to make the following conjecture.
\begin{Conjecture}
Let $E_{/\Q}$ be an elliptic curve of rank $r_E$ with good (ordinary or supersingular) reduction at $p$.
For $100\%$ of the primes, $\mu_p^{\ddag}(E)=0$ and $\lambda_p^{\ddag}(E)=r_E$.
\end{Conjecture}

\begin{Remark}
Let $E_{/\Q}$ be an elliptic curve with $r_E\geq 1$. 
If $p$ is a prime at which the $p$-adic regulator has valuation equal to $p^{r_E}$, the same argument as Remark~\ref{remark: trivial rank every layer} shows that the rank of the elliptic curve remains unchanged in every layer of the cyclotomic $\Z_p$-extension.
\end{Remark}


\section{Results for a fixed prime and varying elliptic curve}
In this section, we fix a prime $p\geq 5$ and study the variation of Iwasawa invariants as $E$ ranges over all elliptic curves of rank zero with good reduction at $p$.
Recall that any elliptic curve $E$ over $\Q$ admits a unique Weierstrass equation
\begin{equation}\label{weier}
E:y^2 = x^3 + Ax + B
\end{equation}
where $A, B$ are integers and $\gcd(A^3 , B^2)$ is not divisible by any twelfth power. 
Since $p\geq 5$, such an equation is minimal. 
We order elliptic curves by height and expect that similar results shall hold when they are ordered by conductor or discriminant.
Recall that the \textit{height of }$E$ satisfying the minimal equation $\eqref{weier}$ is given by $H(E) := \max\left(\abs{A}^3, \abs{B}^2\right)$.

Let $\mathcal{E}$ be set of isomorphism classes of all elliptic curves over $\Q$.
Let $\mathcal{J}$ be the set of elliptic curves $E$ over $\Q$ satisfying the following two properties 
\begin{enumerate}
 \item $E$ has rank zero,
 \item $E$ has good reduction (ordinary or supersingular) at $p$.
\end{enumerate}
The set $\mathcal{J}$ is the (disjoint) union of two sets $\mathcal{J}^{\ord}$ and $\mathcal{J}^{\op{ss}}$, consisting of rank 0 elliptic curves with ordinary and supersingular reduction at $p$, respectively.
For $X>0$, write $\mathcal{E}(X)$ for the set of isomorphism classes of elliptic curves over $\Q$ of height $<X$.
If $\mathcal{S}$ is a subset of $\mathcal{E}$, set $\mathcal{S}(X)=\mathcal{S}\cap \mathcal{E}(X)$.
It is conjectured that when ordered by height, discriminant or conductor, half of the elliptic curves over $\Q$ have rank $0$ (see for example \cite[Conjecture B]{Gol79} or \cite[p. 15]{KS99}).
If $E\in \mathcal{J}$ has good ordinary reduction at $p$, then the Euler characteristic formula $\eqref{ecf1}$ states that \begin{equation}\label{ecfrdinary}\chi(\Gamma, E[p^{\infty}])\sim \frac{\# \Sha(E/\Q)[p^{\infty}]\times \left(\prod_{l}c_l(E)\right)\times \left(\# \widetilde{E}(\F_p)\right)^2}{\#\left(E(\Q)[p^{\infty}]\right)^2}.\end{equation}
Note that in the above equation, $\#E(\Q)[p^{\infty}]=1$ if $p\geq 11$.
On the other hand, if $E\in \mathcal{J}$ has good supersingular reduction at $p$, then by $\eqref{ecf2}$, we have that
\begin{equation}\label{ecfss}\chi^{\pm}(\Gamma, E[p^{\infty}])\sim \# \Sha(E/\Q)[p^{\infty}]\times \left(\prod_{l}c_l(E)\right).\end{equation}
Denote by $c_l^{(p)}(E)$ the $p$-part of $c_l(E)$, given by $c_l^{(p)}(E):=p^{v_p(c_l(E))}$.
The key observation in this section is that to analyze the variation of the Euler characteristic (and hence $\mu$ and $\lambda$-invariants) of elliptic curves, it suffices to study the average behaviour of the following quantities for fixed $p$ and varying $E\in \mathcal{J}$,
\begin{enumerate}
    \item $s_p(E):=\#\Sha(E/\Q)[p^{\infty}]$, 
    \item $\tau_p(E):=\prod_l c_l^{(p)}(E)$,
    \item $\delta_p(E):=\#\left(\widetilde{E}(\F_p)[p]\right)$.
\end{enumerate}
\begin{Definition}Let $\mathcal{E}_1(X)$, $\mathcal{E}_2(X)$, and $\mathcal{E}_3(X)$ be the subset of elliptic curves in $\mathcal{E}(X)$ for which $p$ divides $s_p(E)$, $\tau_p(E)$ and $\delta_p(E)$ respectively.
\end{Definition}
Note that no assumptions are made on the rank of elliptic curves in $\mathcal{E}(X)$ or $\mathcal{E}_i(X)$.
On the other hand, for elliptic curves $E\in \mathcal{J}$, the rank is zero.

The primary goal is to obtain upper bounds for 
\[\mathfrak{d}_p^{(i)}:=\limsup_{X\rightarrow \infty} \frac{\#\mathcal{E}_i(X)}{\#\mathcal{E}(X)}\]
for $i=2,3$ (with no constraints on the rank of the elliptic curves).

In \cite{Del01}, C. Delauney gave heuristics for the average number of elliptic curves with $s_p(E)\neq 1$.
These heuristics are stated in terms of elliptic curves ordered by conductor.
However, they indicate that $\mathfrak{d}_p^{(1)}$ goes to $0$ as $p\rightarrow \infty$ rather fast.
Since there is still not much known about this particular question, we are unable to make further clarifications about the behaviour of $\mathfrak{d}_p^{(1)}$. 
However, we expect that the analysis of this part of the formula is the most difficult.

Let $\kappa=(a,b)\in \F_p\times \F_p$ be such that the discriminant $\Delta(\kappa):=4a^3+27b^2$ is nonzero. 
The elliptic curve $E_{\kappa} :  y^2=x^3+ax+b$ defined over $\F_p$ is smooth.
Let $d(p)$ be the number of pairs $\kappa=(a,b)\in \F_p\times \F_p$ such that
\begin{enumerate}
 \item $\Delta(\kappa)\neq 0$.
 \item $E_{\kappa} : y^2=x^3+ax+b$ has a point over $\F_p$ of order $p$.
\end{enumerate}
For the primes $p$ in the range $5\leq p<500$, computations on sage show that $d(p)\leq 1$ and $d(p)=1$ for $p\in \{5,7,61\}$.
We remark that $d(p)$ is closely related to the Kronecker class number of $1-4p$ (see \cite[p. 184]{Schoof87}).  
The estimate $\eqref{cremonasadek}$ follows from the method of M. Sadek \cite{sadek17}, or the results of J. Cremona and Sadek, see \cite{CS20}.
\begin{Th}\label{theorem5}
Let $p\geq 5$ be a fixed prime number.
Then
\begin{equation}\label{cremonasadek}
\mathfrak{d}_p^{(2)}\leq \sum_{l\neq p} \frac{(l-1)^2}{l^{p+2}},
\end{equation}
where the sum is taken over prime numbers $l\neq p$,
and
\begin{equation}
\label{to be proven in 4.14}
\mathfrak{d}_p^{(3)}\leq \zeta(10)\cdot\frac{d(p)}{p^2}.
\end{equation}
\end{Th}
We provide a proof of \eqref{cremonasadek} in Theorem \ref{the dp2 estimate} and of \eqref{to be proven in 4.14} in Theorem \ref{last result}.
Note that $\zeta(10)=\frac{\pi^{10}}{93555}$ is approximately equal to $1.001$.
This quantity arises since the proportion of Weierstrass equations ordered by height which are minimal is $\frac{1}{\zeta(10)}$ (see \cite{CS20}).
To avoid confusion, we state the results for good ordinary and good supersingular elliptic curves separately.
First, we state the result for elliptic curves with good ordinary reduction at $p$.

\begin{Th}\label{theorem6}
Let $p\geq 5$ be a fixed prime number.
Let $\mathcal{Z}^{\ord}$ denote the set of rank 0 elliptic curves $E$ with good ordinary reduction at $p$ for which the following equivalent conditions are satisfied
\begin{enumerate}
 \item $\chi(\Gamma, E[p^{\infty}])=1$,
 \item $\Selp=0$.
\end{enumerate}
Then, 
\[\limsup_{X\rightarrow \infty} \frac{\#\mathcal{Z}^{\ord}(X)}{\#\mathcal{E}(X)}
\geq  \limsup_{X\rightarrow \infty} \frac{\#\mathcal{J}^{\ord}(X)}{\#\mathcal{E}(X)}-\mathfrak{d}_p^{(1)}-\sum_{l\neq p} \frac{(l-1)^2}{l^{p+2}}-\zeta(10)\cdot \frac{d(p)}{p^2}.\]
\end{Th}
\begin{proof}
It follows from Corollary $\ref{corollary35}$ that $\chi(\Gamma, E[p^{\infty}])=1$ and $\Selp$ are equivalent.
By the Euler characteristic formula $\eqref{ecfrdinary}$,
\[\limsup_{X\rightarrow \infty} \frac{\#\mathcal{Z}^{\ord}(X)}{\#\mathcal{E}(X)}
\geq \limsup_{X\rightarrow \infty} \frac{\#\mathcal{J}^{\ord}(X)}{\#\mathcal{E}(X)}- \mathfrak{d}_p^{(1)}-\mathfrak{d}_p^{(2)}-\mathfrak{d}_p^{(3)}.\]
The result follows from Theorem $\ref{theorem5}$.
\end{proof}

Next, we prove an analogous result in the case when $E$ varies over elliptic curves with good supersingular reduction at $p$.
\begin{Th}\label{supersingulartheorem}
Let $p\geq 5$ be a fixed prime number.
Let $\mathcal{Z}^{\op{ss}}$ be the rank 0 elliptic curves $E$ with good supersingular reduction at $p$, for which the following equivalent conditions are satisfied
\begin{enumerate}
 \item $\chi^{\pm}(\Gamma, E[p^{\infty}])=1$
 \item $\Selpm=0$.
\end{enumerate}
Then,
\[\limsup_{X\rightarrow \infty} \frac{\#\mathcal{Z}^{\op{ss}}(X)}{\#\mathcal{E}(X)}
\geq \limsup_{X\rightarrow \infty} \frac{\#\mathcal{J}^{\op{ss}}(X)}{\#\mathcal{E}(X)}-\mathfrak{d}_p^{(1)}-\sum_{l\neq p} \frac{(l-1)^2}{l^{p+2}}.\]
\end{Th}
\begin{proof}
The proof is identical to that of Theorem $\ref{theorem6}$. 
It is a direct consequence of Corollary $\ref{corollary35}$, Theorem $\ref{theorem5}$, and the Euler characteristic formula $\eqref{ecfss}$.
\end{proof}
We prove a result which applies for all elliptic curves with good reduction.
\begin{Th}
\label{main result rank 0 good reduction}
Let $p\geq 5$ be a fixed prime number.
Let $\mathcal{Z}$ be the set of rank 0 elliptic curves $E$ with good reduction at $p$, for which $\Sel_{p^\infty}^{\ddag}(E/\Q^{\cyc})=0$.
Then,
\[\limsup_{X\rightarrow \infty} \frac{\#\mathcal{Z}(X)}{\#\mathcal{E}(X)}
\geq \limsup_{X\rightarrow \infty} \frac{\#\mathcal{J}(X)}{\#\mathcal{E}(X)}-\mathfrak{d}_p^{(1)}-\sum_{l\neq p} \frac{(l-1)^2}{l^{p+2}}-\zeta(10)\cdot \frac{d(p)}{p^2}.\]
\end{Th}
\begin{proof}
Let $\mathcal{Y}\subset \mathcal{E}$ consist of the elliptic curves $E$ for which $s_p(E)\tau_p(E)\delta_p(E)\neq 1$. It follows from the Euler characteristic formulas $\eqref{ecfrdinary}$ and $\eqref{ecfss}$ that $\mathcal{Z}$ is contained in $\mathcal{E}\setminus \mathcal{Y}$.
By Theorem $\ref{theorem6}$ that 
\[\limsup_{X\rightarrow \infty} \frac{\#\mathcal{Y}(X)}{\#\mathcal{E}(X)}\leq \mathfrak{d}_p^{(1)}+\sum_{l\neq p} \frac{(l-1)^2}{l^{p+2}}+\zeta(10)\cdot \frac{d(p)}{p^2}\] and the result follows.
\end{proof}

\begin{Remark}
On average, the proportion of elliptic curves over $\Z_p$ with good reduction at $p$ (ordered by height) is $(1-\frac{1}{p})$, see \cite{CS20}.
Also, it is expected that $1/2$ the elliptic curves have rank $1$ when ordered by height.
Therefore, it is reasonable to expect that
\[\limsup_{X\rightarrow \infty}\frac{\#\mathcal{J}(X)}{\#\mathcal{E}(X)}=\frac{1}{2}\left(1-\frac{1}{p}\right).\]
\end{Remark}
Heuristics of Delauney suggest that $\mathfrak{d}_p^{(1)}$ should approach zero quite rapidly as $p\rightarrow \infty$.
The result indicates that the proportion of elliptic curves for which the Selmer group is zero is $>0$ and the proportion approaches $1/2$ as $p\rightarrow \infty$.
\par We are led to make the following conjecture.
\begin{Conjecture}\label{lastconj}
Let $p$ be a fixed prime. 
Denote by $\mathcal{J}_p$ the set of rank 0 elliptic curves with good reduction at $p$, and by $\mathcal{Z}_p$ the subset of elliptic curves for which $\Sel_{p^\infty}^{\ddag}(E/\Q^{\cyc})=0$.
Then,
\[\liminf_{p\rightarrow \infty} \left(\limsup_{X\rightarrow \infty} \frac{\mathcal{Z}_p(X)}{\mathcal{J}_p(X)} \right)=1.\]
\end{Conjecture}

\begin{Remark}
In the rank one case, such an analysis is difficult.
This is because of the term arising from the $p$-adic regulator in the formula for the truncated Euler characteristic.
At the time of writing, the authors are not aware of any results or heuristics for the average behaviour of the $p$-adic valuation of $R_p(E/\Q)$ as $E$ ranges over all elliptic curves of rank $1$ with good ordinary reduction at $p$.
\end{Remark}

Theorem $\ref{theorem5}$ is proved in the remainder of the section.
\subsection{Average results on Tamagawa numbers}
Let $p\geq 5$ be a fixed prime, and $l$ be a prime different from $p$.
In this section, we estimate the proportion of elliptic curves $E/\Q$ up to height $X$ with Kodaira type $I_{p}$ at $l$. 
These estimates are well known, but we include them for the sake of completeness, see \cite{sadek17,CS20}.
Recall that when the Kodaira symbol at the prime $l$ is $I_{p}$, the Tamagawa number $c_l$ is divisible by $p$ \cite[p. 448]{Sil09}.
Let $\cE(X)$ be the set of isomorphism classes of all elliptic curves over $\Q$ with height $\leq X$.
This is in one-to-one correspondence with the set
\[
\begin{Bmatrix}
 & \abs{A}\leq \sqrt[3]{X}, \ \abs{B}\leq \sqrt{X}\\
(A,B) \in \Z \times \Z: & \quad 4A^3 +27B^2 \neq 0 \\
& \textrm{for all primes } q \textrm{ if } q^4 |A, \textrm{then } q^6\nmid B 
\end{Bmatrix}.
\]

\begin{Lemma}[A. Brumer]
\label{Brumer}
With notation as above, 
\[
\# \cE(X) = \frac{4X^{5/6}}{\zeta(10)} + O\left(\sqrt{X}\right).
\]
\end{Lemma}
\begin{proof}
See {\cite[Lemma 4.3]{Bru92}}.
\end{proof}

Consider the set $\cE^{I_{p}}_l(X)$, i.e. the set of elliptic curves over $\Q$ with bad reduction at $l$, height $\leq X$, and Kodaira type $I_{p}$.
The Kodaira symbol forces the bad reduction to be of multiplicative type.
It follows from Tate's algorithm that this set is in one-to-one correspondence with
\begin{equation}
\label{correspondence}
\begin{Bmatrix}
 & \abs{A}\leq \sqrt[3]{X}, \ \abs{B}\leq \sqrt{X}\\
(A,B) \in \Z \times \Z: & l\nmid A, \ l\nmid B, \ l^{p}\edv 4A^3 + 27B^3 \\
& (A,B)\neq (0,0)\in \Z/q^4 \times \Z/q^6 \textrm{ for any prime }q 
\end{Bmatrix}.
\end{equation}
We include both upper and lower bounds, however, we only apply upper bounds in our analysis.
The following calculations have been done in the preprint \cite[Lemma 4.1]{sadek17}. We clarify the arguments and include them here for completeness. 
\begin{Lemma}[Sadek]
Let $X>0$ and $l_i$ be the $i$-th prime.
Let $k>0$ be the largest positive integer such that $L_k = \prod_{i=1}^{k} l_i \leq \sqrt[12]{X}$.
Then, 
\begin{align*}
 4l^{p}(l-1)^2\prod_{i=1}^{k}(l_i^{10}-1)\left(\floor*{\frac{\sqrt[3]{X}}{l^{p+1} L_k^{4}}} \floor*{\frac{\sqrt{X}}{l^{p+1} L_k^6}} - \frac{X^{5/6}}{9l^{2p+2} L_k^{10} l_k^9}\right) \leq \# \cE_l^{I_{p}}(X) &\\
 \leq 4l^{p}(l-1)^2\prod_{i=1}^{k}(l_i^{10}-1)\left(\floor*{\frac{\sqrt[3]{X}}{l^{p+1} L_k^{4}}} \floor*{\frac{\sqrt{X}}{l^{p+1} L_k^6}} + \frac{\sqrt[3]{X}}{3l^{p+1} L_k^{4} l_k^3} + \frac{\sqrt{X}}{5l^{p+1} L_k^{6} l_k^5}\right)&.
\end{align*}
\end{Lemma}

\begin{proof}
To obtain the estimate on the size of the set $\cE_l^{I_p}(X)$, we use the description of the set in \eqref{correspondence}.
Observe that $l^{p+1} \nmid 4A^3 + 27B^2$, hence the condition $4A^3 + 27B^2 \neq 0$ is inherent in the definition of $\cE_l^{I_p}(X)$.

Consider the congruence equation $4A^3 + 27B^2 \equiv 0 \pmod{l}$.
It has $l-1$ non-singular solutions, which lift to $l^{p-1}(l-1)$ solutions modulo $l^{p}$ (see for example \cite[\S 3.4.1]{Wat08}).
Note that the description of the set in \eqref{correspondence} says that $l\nmid A$ and $l\nmid B$; thereby allowing us to ignore the point $(0,0)$. 

In view of \eqref{correspondence}, we are interested in solutions modulo $l^{p}$ that fail to satisfy the congruence equation modulo $l^{p+1}$. Since the $l-1$ non-singular solutions lift to $l^{p-1}(l-1)$ solutions modulo $l^p$ and $l^p(l-1)$ solutions modulo $l^{p+1}$, it follows that the number of pairs $(A,B)\in \Z/l^{p+1}\times \Z/l^{p+1}$ such that
\begin{enumerate}
    \item $(A,B)\not \equiv (0,0)\mod{l}$,
    \item $4A^3+27B^2\equiv 0\mod{l^p}$,
    \item $4A^3+27B^2\not \equiv 0\mod{l^{p+1}}$,
\end{enumerate}
is equal to 
\[l^2\cdot (l^{p-1}(l-1))-l^p(l-1)=l^p(l-1)^2.\]
Therefore, $\cE_l^{I_{p}}(X)$ has $l^p(l-1)^2$ pairs of residue classes in $\Z/l^{p+1}\times \Z/l^{p+1}$.
Next, we need to count the number of lifts of each such pair under the additional condition that $(A,B)\neq (0,0)\in \Z/l_i^4 \times \Z/l_i^6$ for each prime $l_i$.
Note that the number of pairs $(A,B)$ satisfying this additional condition is $(l_i^{10}-1)$.

It follows that the number of pairs $(A,B)$ in the box $[-\sqrt[3]{X},\sqrt[3]{X}]\times[-\sqrt{X},\sqrt{X}]$ such that $(A,B)\neq (0,0)\in \Z/l^{p+1} \times \Z/l^{p+1}$, $(A,B)\neq (0,0)\in \Z/l_i^4 \times \Z/l_i^6$ and $4A^3 + 27B^2 \equiv 0 \mod l^p$ is
\[
 4l^{p}(l-1)^2 \prod_{i=1}^k\left( l_i^{10}-1\right)\floor*{\frac{\sqrt[3]{X}}{l^{p+1} L_k^4}}\floor*{\frac{\sqrt{X}}{l^{p+1} L_k^6}}.
\]
Our estimate so far might include pairs $(A,B)$ such that it is $(0, 0) \in \Z/q^4 \times \Z/q^6$ when $l_k < q \leq \sqrt[12]{X}$.
So, we must exclude the integral pairs which reduce to $(0,0)\in \Z/q^4 \times \Z/q^6$ for $l_k < q \leq \sqrt[12]{X}$.
Therefore, we need to remove 
\[
4l^{p}(l-1)^2 \prod_{i=1}^k\left( l_i^{10}-1\right)\sum_{l_k <q\leq \sqrt[12]{X} }\floor*{\frac{\sqrt[3]{X}}{l^{p+1} q^4 L_k^4}}\floor*{\frac{\sqrt{X}}{l^{p+1} q^6 L_k^6}}
\]
many pairs from our count.
Putting this together, we get 
\begin{align*}
\# \cE_l^{I_{p}}(X) & = 4l^{p}(l-1)^2 \prod_{i=1}^k\left( l_i^{10}-1\right)\floor*{\frac{\sqrt[3]{X}}{l^{p+1} L_k^4}}\floor*{\frac{\sqrt{X}}{l^{p+1} L_k^6}}\\
& - 4l^{p}(l-1)^2 \prod_{i=1}^k\left( l_i^{10}-1\right)\sum_{l_k <q\leq \sqrt[12]{X} }\floor*{\frac{\sqrt[3]{X}}{l^{p+1} q^4 L_k^4}}\floor*{\frac{\sqrt{X}}{l^{p+1} q^6 L_k^6}}.
\end{align*}
To obtain the expression as in the statement of the lemma, we will need to manipulate the following term
\[
\sum_{l_k <q\leq \sqrt[12]{X} }\floor*{\frac{\sqrt[3]{X}}{l^{p+1} q^4 L_k^4}}\floor*{\frac{\sqrt{X}}{l^{p+1} q^6 L_k^6}}.
\]
First we consider its upper bound.
\begin{align*}
    \sum_{l_k <q\leq \sqrt[12]{X} }\floor*{\frac{\sqrt[3]{X}}{l^{p+1} q^4 L_k^4}}\floor*{\frac{\sqrt{X}}{l^{p+1} q^6 L_k^6}} & \leq \frac{X^{5/6}}{l^{2(p+1)}L_k^{10}}\sum_{l_k < q \leq \sqrt[12]{X}} 1/q^{10}\\
    & \leq \frac{X^{5/6}}{l^{2(p+1)}L_k^{10}} \int_{l_k}^{\infty} 1/x^{10} dx\\
    & \leq  \frac{X^{5/6}}{9 l^{2(p + 1)}l_k^9 L_k^{10}}.
\end{align*}
Now for the lower bound, observe that
\begin{align*}
    \sum_{l_k <q\leq \sqrt[12]{X} }\floor*{\frac{\sqrt[3]{X}}{l^{p+1} q^4 L_k^4}}\floor*{\frac{\sqrt{X}}{l^{p+1} q^6 L_k^6}} & \geq \sum_{l_k < q \leq \sqrt[12]{X}} \left( \frac{\sqrt[3]{X}}{l^{p+1}q^4 L_k^4} - 1 \right)\left( \frac{\sqrt{X}}{l^{p+1}q^6 L_k^6} - 1 \right)\\
    & \geq - \frac{\sqrt[3]{X}}{l^{p+1}L_k^4} \sum_{l_k < q \leq \sqrt[12]{X}} 1/q^4 - \frac{\sqrt{X}}{l^{p+1}L_k^6} \sum_{l_k < q \leq \sqrt[12]{X}} 1/q^6\\
    & \geq -\frac{\sqrt[3]{X}}{3l^{p+1} L_k^{4} l_k^3} - \frac{\sqrt{X}}{5l^{p+1} L_k^{6} l_k^5}.
\end{align*}    
The result is now immediate.
\end{proof}
The following result follows from the previous lemmas in this section.
\begin{Th}
\label{the dp2 estimate}
With notation as above,
\begin{align*}
\limsup_{X\rightarrow \infty} \frac{\# \cE^{I_{p}}_l(X)}{\# \cE(X)} &\leq \frac{(l-1)^2}{l^{p + 2}}.
\end{align*}
\end{Th}

\begin{proof}
The result follows from using the upper bound of $\cE_l^{I_p}(X)$ and Lemma \ref{Brumer}. 
\end{proof}

\begin{Remark}
When the elliptic curves are ordered by conductor (rather than height), the same bounds have been obtained in \cite[Theorem 1.6]{SSW19+}.
\end{Remark}

\subsection{Average results on anomalous primes}
We fix a prime $p\geq 5$.
Let $\mathcal{W}$ consist of tuples $(A,B)\in \Z\times \Z$, where $(A,B)$ is identified with the (minimal) Weierstrass equation 
\[y^2=x^3+Ax+B.\]
Denote by $\mathcal{W}(X)$ the set of Weierstrass equations for which the height is $\leq X$.
Note that $\mathcal{E}(X)$ is a subset of $\mathcal{W}(X)$ such that 
\begin{equation}
\label{99.9 are glob min}
\lim_{X\rightarrow \infty} \frac{\mathcal{E}(X)}{\mathcal{W}(X)}=\frac{1}{\zeta(10)}
\end{equation}
(see \cite{CS20}). 
Thus $99.9\%$ of Weierstrass equations are globally minimal.

Let $\kappa=(a,b)\in \F_p\times \F_p$ with $\Delta(\kappa)\neq 0$. 
Let $E_{\kappa}$ be the elliptic curve defined by the Weierstrass equation \[E_{\kappa}:y^2=x^3+ax+b.\]
Note that $\kappa$ is not uniquely determined by $E_{\kappa}$.

\begin{Lemma}\label{lastlemma}
Let $\kappa$ be a pair and $\mathcal{W}_{\kappa}(X)\subset \mathcal{W}(X)$ be the subset of Weierstrass equations $y^2=x^3+Ax+B$ such that the pair $(A,B)$ reduces to $\kappa$.
Then, 
\[\limsup_{X\rightarrow \infty}\frac{\mathcal{W}_{\kappa}(X)}{\mathcal{E}(X)}\leq \frac{\zeta(10)}{p^2}.\]
\end{Lemma}
\begin{proof}
Observe that $\mathcal{W}_{(0,0)}$ is the lattice in $\Z\times \Z$ with lattice basis $(p,0)$ and $(0,p)$.
Since $\mathcal{W}_{\kappa}$ is simply a translation of $\mathcal{W}_{(0,0)}$, it follows that
\[\lim_{X\rightarrow \infty} \frac{\mathcal{W}_{\kappa}(X)}{\mathcal{W}(X)}=\frac{1}{p^2}.\]
The result follows from $\eqref{99.9 are glob min}$.
\end{proof}
Denote by $\mathfrak{S}$ the set of pairs $\kappa=(a,b)\in \F_p\times \F_p$ such that $E_{\kappa}$ contains a point of order $p$ over $\F_p$.
Recall that $d(p):=\# \mathfrak{S}$.
Let $\mathcal{W}'(X)\subset \mathcal{W}(X)$ be the set of Weierstrass equations $y^2=x^3+Ax+B$ which reduce to $E_{\kappa}$ for some $\kappa\in \mathfrak{S}$.
\begin{Th}
\label{last result}
We have that $\mathfrak{d}_p^{(3)}\leq \zeta(10)\cdot \frac{d(p)}{p^2}$.
\end{Th}
\begin{proof}
It follows from Lemma $\ref{lastlemma}$ that 
\[\limsup_{X\rightarrow \infty}\frac{\mathcal{W}'(X)}{\mathcal{E}(X)}\leq \zeta(10)\cdot\frac{d(p)}{p^2}.\]
Recall that $\mathcal{E}_3(X)\subseteq \mathcal{W}'(X)$.
The result follows.
\end{proof}

\section*{Acknowledgments}
DK thanks J. Balakrishnan, I. Varma, and N. Kaplan for helpful discussions.
She acknowledges the support of the PIMS Postdoctoral Fellowship.
AR is grateful to R. Sujatha for informing him of the compelling role of the generalized Euler characteristic in Iwasawa theory.
We thank the referee for timely reading of the paper and expert suggestions.

\bibliographystyle{abbrv}
\bibliography{references}
\end{document}